\documentclass{article}

\usepackage{amsmath,amsthm,bm}
\usepackage{enumerate}
\usepackage{comment}
\usepackage{amssymb}
\usepackage{pgfplots}
\usepackage{subcaption}
\usepackage{graphicx}
\usepackage{url}
\usetikzlibrary{arrows,positioning,calc,patterns}

\newtheorem{thm}{Theorem}[section]
\newtheorem{cor}[thm]{Corollary}
\newtheorem{conj}[thm]{Conjecture}
\newtheorem{prop}[thm]{Proposition}
\newtheorem{lem}[thm]{Lemma}

\begin{document}

\title{On the Area of the Fundamental Region of a Binary Form Associated with Algebraic Trigonometric Quantities}

\author{Anton Mosunov}

\date{}

\maketitle

\begin{abstract}
Let $F(x, y)$ be a binary form of degree at least three and non-zero discriminant. We estimate the area $A_F$ bounded by the curve $|F(x, y)| = 1$ for four families of binary forms. The first two families that we are interested in are homogenizations of minimal polynomials of $2\cos\left(\frac{2\pi}{n}\right)$ and $2\sin\left(\frac{2\pi}{n}\right)$, which we denote by $\Psi_n(x, y)$ and $\Pi_n(x, y)$, respectively. The remaining two families of binary forms that we consider are homogenizations of Chebyshev polynomials of the first and second kinds, denoted $T_n(x, y)$ and $U_n(x, y)$, respectively.
\end{abstract}

\section{Introduction}
Let $F(x, y)$ be a binary form with real coefficients of degree $n \geq 3$ and non-zero discriminant $D_F$. The set
$$
\{(x, y) \in \mathbb R^2 \colon |F(x, y)| \leq 1\},
$$
which represents the collection of all points bounded by the curve $|F(x, y)| = 1$, is called the \emph{fundamental region} of $F$. It was proved by Mahler \cite{mahler33} that the area $A_F$ of the fundamental region of $F$ is finite.

In what follows, we restrict our attention to binary forms $F$ having integer coefficients. The quantity $A_F$, which can be evaluated via the formula \mbox{\cite[Section 3]{bean94}}
\begin{equation} \label{eq:AF-formula}
A_F = \int\limits_{-\infty}^{+\infty}\frac{dx}{|F(x, 1)|^{2/n}},
\end{equation}
plays a significant role in the analysis of certain Diophantine equations and inequalities associated with $F$. Mahler \cite{mahler33} proved that, for a positive \mbox{integer $h$}, the number of integer solutions $Z_F(h)$ to the \emph{Thue inequality} $|F(x, y)| \leq h$ satisfies
$$
\left|Z_F(h) - A_Fh^{2/n}\right| \ll_F h^{1/(n - 1)}.
$$
In 2019, Stewart and Xiao \cite{stewart-xiao} proved that the number of integers of absolute value at most $h$ which are represented by the form $F$ is asymptotic to $C_Fh^{2/n}$, where $C_F$ is a positive number which depends on $F$ and is a rational multiple of $A_F$.

The values $A_F$ for certain binary forms $F$ are intimately connected to the values of the \emph{beta function}
\begin{equation} \label{eq:beta-reg}
B(x, y) = \int\limits_0^{1}t^{x - 1}(1 - t)^{y - 1}dt,
\end{equation}
where $x$ and $y$ are complex numbers with positive real parts. An overview of important properties of $B(x, y)$ is given in Section \ref{sec:beta}. For a matrix $M = \left(\begin{smallmatrix}a & b\\c & d\end{smallmatrix}\right)$, with real coefficients, define
$$
F_M(x, y) = F(ax + by, cx + dy).
$$
Two forms $F$ and $G$ are said to be \emph{equivalent under} $\operatorname{GL}_2(\mathbb R)$ if $G = F_M$ for some $M \in \operatorname{GL}_2(\mathbb R)$. Analogously, we define the equivalence of forms under $\operatorname{GL}_2(\mathbb Z)$. In 1994, Bean \mbox{\cite[Corollary 1]{bean94}} proved that $A_F = 3B\left(\frac{1}{3}, \frac{1}{3}\right) = 15.8997\ldots$ when $F$ is a binary form that is equivalent under $\operatorname{GL}_2(\mathbb Z)$ to $xy(x - y)$. Furthermore, he proved that this value is the largest among all binary forms with integer coefficients, degree at least three and non-zero discriminant.

It was proved by Stewart and Xiao \cite[Corollary 1.3]{stewart-xiao} that if $a, b$ are fixed non-zero integers and $F_n(x, y) = ax^n + by^n$ is a binomial form of degree $n \geq 3$, then
\begin{equation} \label{eq:AFn}
A_{F_n} = \begin{cases}
\frac{1}{n|ab|^{1/n}}\left(2B\left(\frac{1}{n},1-\frac{2}{n}\right) +B\left(\frac{1}{n},\frac{1}{n}\right)\right) & \text{if $n$ is odd,}\\
\frac{2}{n|ab|^{1/n}}B\left(\frac{1}{n},\frac{1}{n}\right) & \text{if $ab$ is positive and $n$ is even,}\\
\frac{4}{n|ab|^{1/n}}B\left(\frac{1}{n},1-\frac{2}{n}\right) & \text{if $ab$ is negative and $n$ is even.}
\end{cases}
\end{equation}
In Section \ref{sec:binomial-form} we prove that $\lim\limits_{n \rightarrow \infty} A_{F_n} = 4$.

In 2019, Fouvry and Waldschmidt \cite[Th\'eor\`eme 1.5]{fouvry-waldschmidt} estimated $A_{\Phi_n}$ for cyclotomic binary forms $\Phi_n$ (note that a positive integer $n$ no longer refers to the degree of the form). It is a consequence of their result that for any $\varepsilon > 0$ there exists $n_0 = n_0(\varepsilon)$ such that for all $n \geq n_0$ the inequalities
$$
\left(2 - n^{-1 + \varepsilon}\right)^2 < A_{\Phi_n} < \left(2 + n^{-1 + \varepsilon}\right)^2
$$
are satisfied. Consequently, $\lim\limits_{n \rightarrow \infty} A_{\Phi_n} = 4$.

In this article we estimate $A_F$ for four families of binary forms. Let $\Psi_n(x)$ and $\Pi_n(x)$ denote the minimal polynomials of $2\cos\left(\frac{2\pi}{n}\right)$ and $2\sin\left(\frac{2\pi}{n}\right)$, respectively. The first two families that we are interested in are $\Psi_n(x, y)$ and $\Pi_n(x, y)$, which are homogenizations of $\Psi_n(x)$ and $\Pi_n(x)$, respectively. By \cite[Lemma]{watkins-zeitlin},
\begin{equation} \label{eq:Psin-formula}
\Psi_n(x, y) = \prod\limits_{\substack{1 \leq k < \frac{n}{2}\\\gcd(k, n) = 1}}\left(x - 2\cos\left(\frac{2\pi k}{n}\right)y\right).
\end{equation}
Further, since $\sin\left(\frac{2\pi}{n}\right) = \cos\left(\frac{2\pi(n - 4)}{4n}\right)$, it follows from (\ref{eq:Psin-formula}) that $\sin\left(\frac{2\pi}{n}\right)$ is an algebraic conjugate of $\cos\left(\frac{2\pi}{c(n)}\right)$, where $c(n)$ is the denominator of $\frac{n - 4}{4n}$ (in lowest terms). Consequently,
$$
\Pi_n(x, y) = \Psi_{c(n)}(x, y).
$$
See Figure \ref{fig:Psi13Pi9} for the graphs of $|\Psi_{13}(x, y)| = 1$ and $|\Pi_9(x, y)| = 1$. Note that in both cases the fundamental regions are \emph{not} compact.

\begin{figure}[t!]
\centering
\includegraphics[width=170pt]{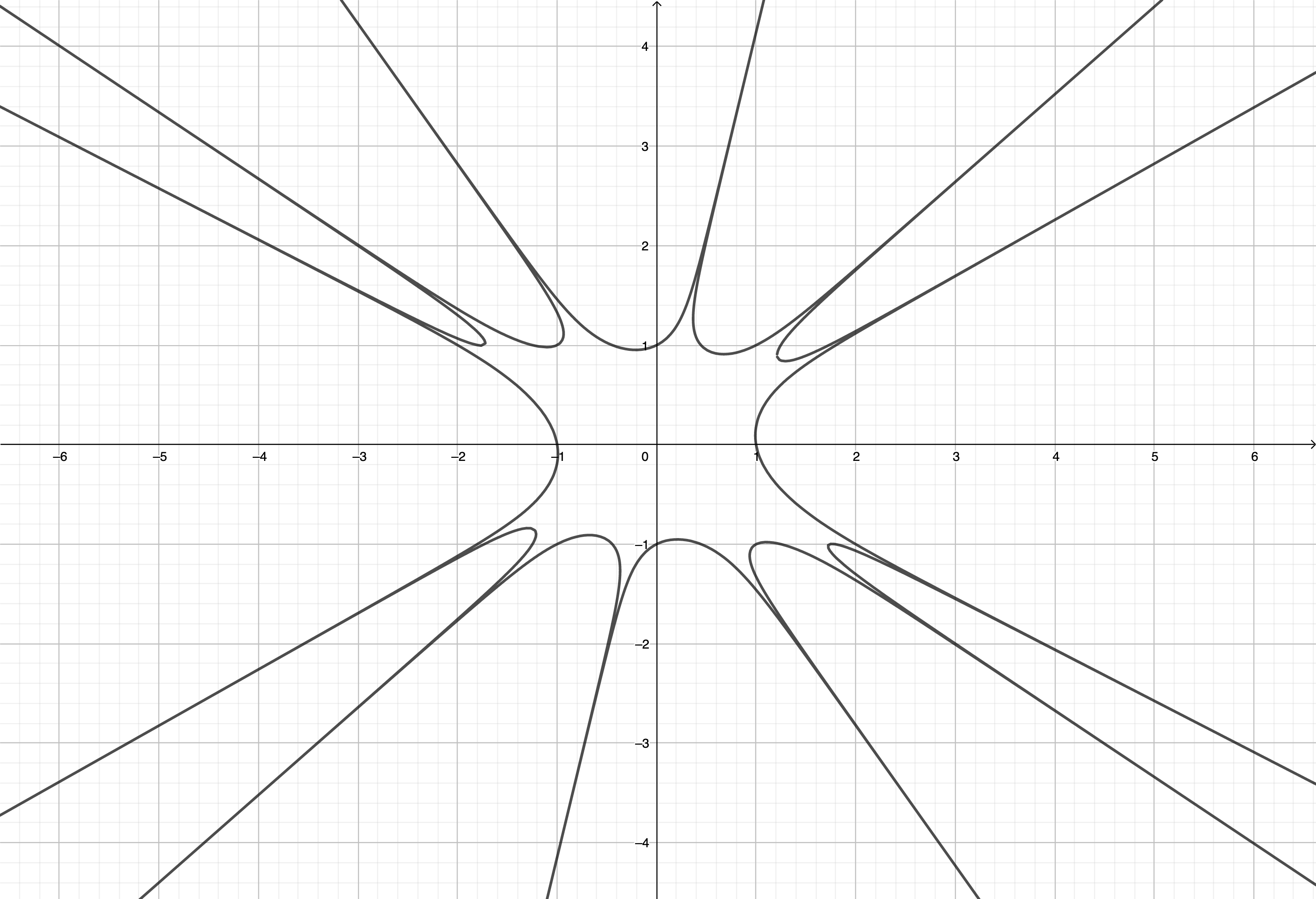}
\includegraphics[width=170pt]{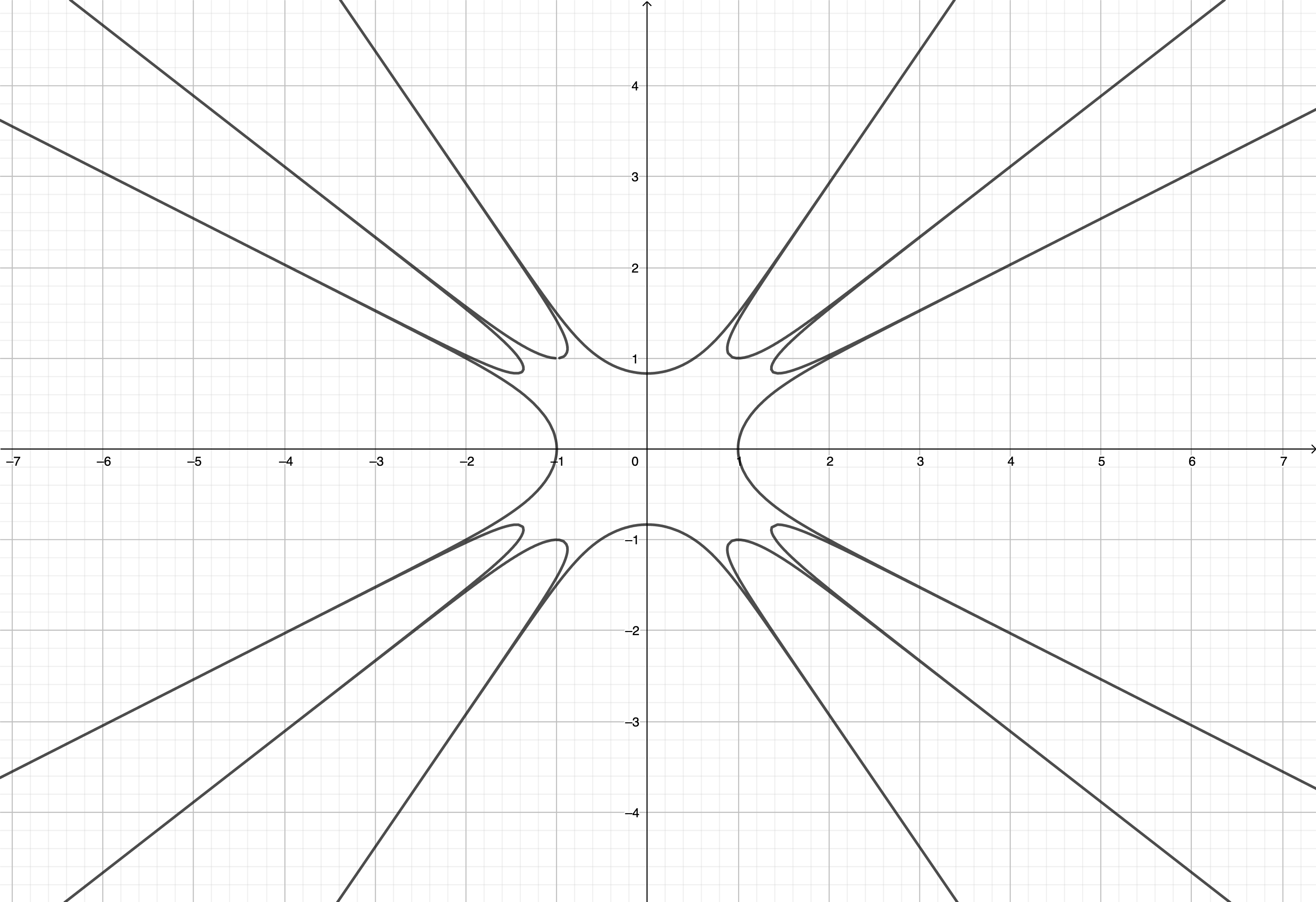}
\caption{Graphs of $|\Psi_{13}(x, y)| = 1$ (left) and $|\Pi_9(x, y)| = 1$ (right)}
\label{fig:Psi13Pi9}
\end{figure}

Next, let $T_n(x)$ and $U_n(x)$ denote Chebyshev polynomials of the first and second kinds, respectively. The other two families that we are interested in are $T_n(x, y)$ and $U_n(x, y)$, which are homogenizations of $T_n(x)$ and $U_n(x)$, respectively. It is known \cite{mason-handscomb} that
$$
T_n(x, y) = 2^{n - 1}\prod\limits_{k = 0}^{n-1}\left(x - \cos\left(\frac{(2k + 1)\pi}{2n}\right)y\right)
$$
and
$$
U_n(x, y) = 2^n\prod\limits_{k = 1}^n\left(x - \cos\left(\frac{k\pi}{n + 1}\right)y\right).
$$
See Figure \ref{fig:T6U6} for the graphs of $|T_6(x, y)| = 1$ and $|U_6(x, y)| = 1$. Note that in both cases the fundamental regions are \emph{not} compact.
%
It is also known that the binary forms $T_n(x, y)$ and $U_n(x, y)$ both have integer coefficients \cite{mason-handscomb}.

For a positive integer $n$, let $d(n)$ denote the number of its positive divisors. Let $\varphi(n)$ denote the Euler's totient function. Our results are summarized in Theorem \ref{thm:main3}, Corollary \ref{cor:main3}, Theorem \ref{thm:main1} and Theorem \ref{thm:main2}.

\newpage

\begin{thm} \label{thm:main3}
Let $n$ be a positive integer such that $\varphi(n) \geq 6$. Let $\Psi_n(x, y)$ denote the homogenization of the minimal polynomial of $2\cos\left(\frac{2\pi}{n}\right)$. Then
\small
\begin{align} \label{eq:APsin-bounds}
A_{\Psi_n} & > \frac{16}{3}\exp\left(-\frac{2d(n)\log n}{\varphi(n)}\right)\\
A_{\Psi_n} & < 2^{\frac{4}{\varphi(n)}}\exp\left(\frac{2d(n)^2\log n}{\varphi(n)}\right) \times\notag\\
& \times \left[\frac{16}{3} + 2^{1-\frac{4}{\varphi(n)}}B\left(\frac{1}{2},\frac{1}{2} - \frac{2}{\varphi(n)}\right) - 2\pi + \left(\frac{2}{n}B\left(\frac{1}{n}, 1 - \frac{4}{\varphi(n)}\right) - 2\right) - \left(\frac{2}{n}B\left(\frac{3}{n}, 1 - \frac{4}{\varphi(n)}\right) - \frac{2}{3}\right)\right].\notag
\end{align}
\normalsize
Furthermore, $\lim\limits_{n \rightarrow \infty} A_{\Psi_n} = \frac{16}{3}$.
\end{thm}

\begin{figure}[t!]
\centering
\includegraphics[width=170pt]{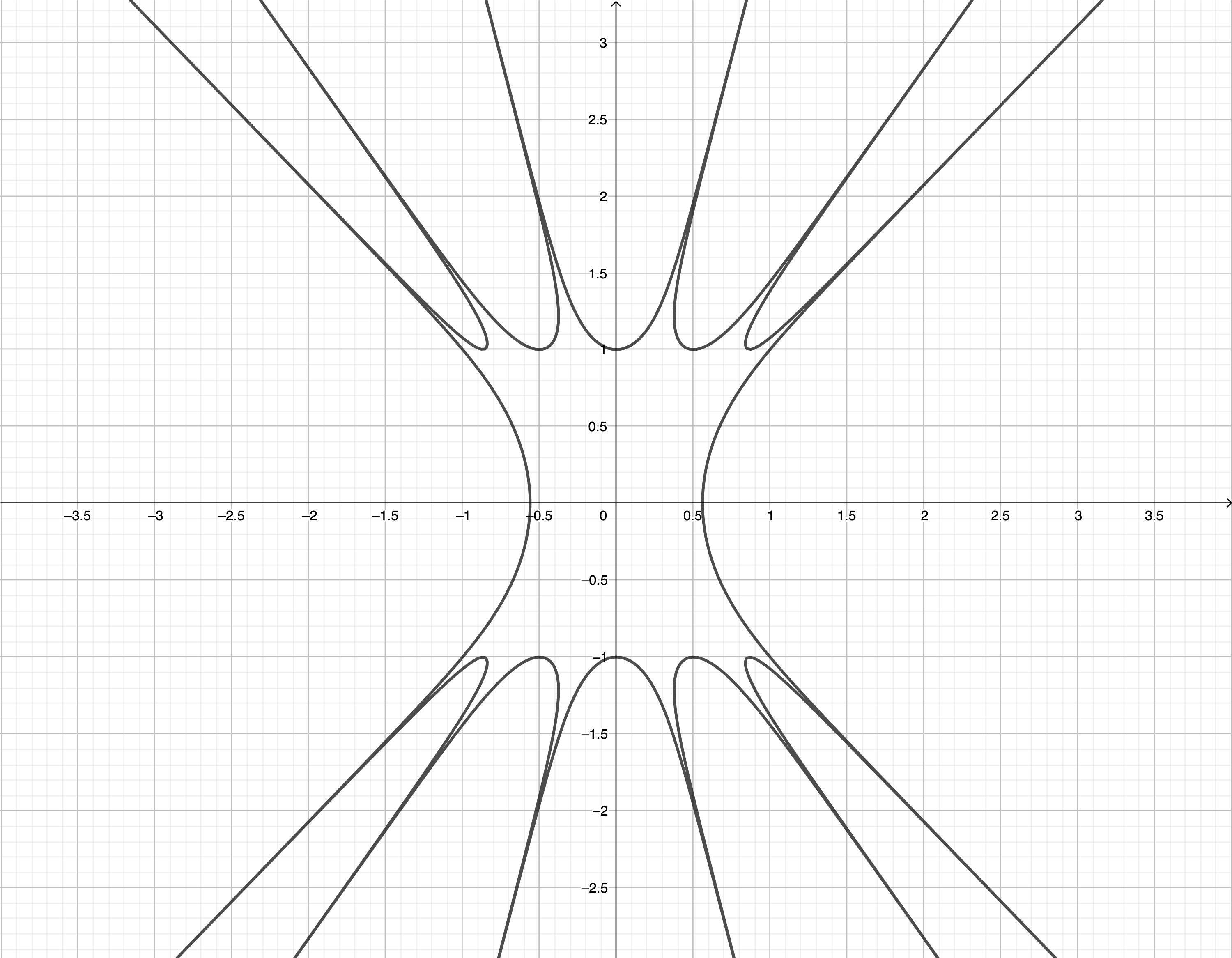}
\includegraphics[width=170pt]{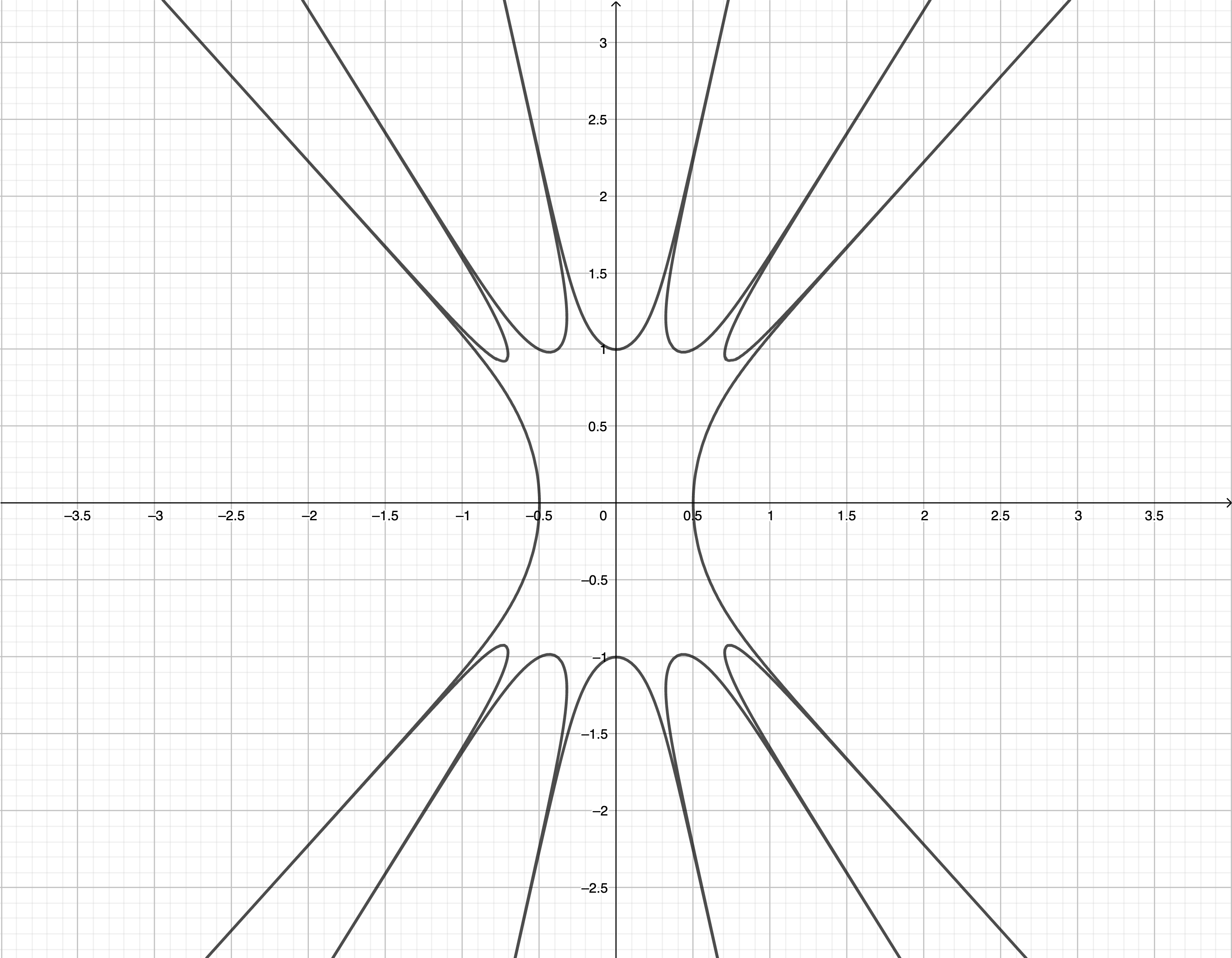}
\caption{Graphs of $|T_6(x, y)| = 1$ (left) and $|U_6(x, y)| = 1$ (right)}
\label{fig:T6U6}
\end{figure}

\begin{cor} \label{cor:main3}
Let $n$ be a positive integer. Let $\Pi_n(x, y)$ denote the homogenization of the minimal polynomial of $2\sin\left(\frac{2\pi}{n}\right)$. Then $\Pi_n(x, y) = \Psi_{c(n)}(x, y)$, where
$$
c(n) =
\begin{cases}
4n & \text{if $n$ is odd,}\\
2n & \text{if $n \equiv 2 \pmod 4$,}\\
n & \text{if $n \equiv 0 \pmod 8$,}\\
n/2 & \text{if $n \equiv 12 \pmod{16}$,}\\
n/4 & \text{if $n \equiv 4 \pmod{16}$.}
\end{cases}
$$
Furthermore, $\lim\limits_{n \rightarrow \infty} A_{\Pi_n} = \frac{16}{3}$.
\end{cor}

\begin{thm} \label{thm:main1}
Let $n$ be a positive integer such that $n \geq 3$. Let $T_n(x, y)$ denote the homogenization of the Chebyshev polynomial of the first kind $T_n(x)$. Then
\begin{equation} \label{eq:ATn-bounds}
\frac{8}{3} < A_{T_n} < \frac{8}{3} + \frac{2}{3}\left(\sqrt[n]{4} - 1\right) + B\left(\frac{1}{2} - \frac{1}{n}, \frac{1}{2}\right) - \pi
\end{equation}
Furthermore, $\lim\limits_{n \rightarrow \infty} A_{T_n} = \frac{8}{3}$.
\end{thm}

\begin{thm} \label{thm:main2}
Let $n$ be a positive integer such that $n \geq 3$. Let $U_n(x, y)$ denote the homogenization of the Chebyshev polynomial of the second kind $U_n(x)$. Then
\begin{align} \label{eq:AUn-bounds}
A_{U_n} & > \frac{8}{3} + \left(B\left(1 + \frac{1}{n}, \frac{1}{2}\right) - 2\right) + \frac{2}{3}\left((n + 1)^{-2/n} - 1\right)\\
A_{U_n} & < \frac{8}{3} + \left(B\left(1 + \frac{1}{n}, \frac{1}{2}\right) - 2\right) + B\left(\frac{1}{2} - \frac{1}{n}, \frac{1}{2}\right) - \pi\notag
\end{align}
Furthermore, $\lim\limits_{n \rightarrow \infty} A_{U_n} = \frac{8}{3}$.
\end{thm}

Theorem \ref{thm:main3} follows from Lemma \ref{lem:psi-area}, which we establish in Section \ref{sec:Psi}. The proof of Corollary \ref{cor:main3} is outlined in Section \ref{sec:pi}. Theorem \ref{thm:main1} follows from Lemma \ref{lem:chebyshev1-area}, which we establish in Section \ref{sec:Tn}. Theorem \ref{thm:main2} follows from Lemma \ref{lem:chebyshev2-area}, which we establish in Section \ref{sec:Un}.

Our last result, which we present in Section \ref{sec:bean}, concerns the quantity
$$
Q(F) = |D_F|^{1/n(n - 1)}A_F
$$
associated to a binary form $F$ of degree $n \geq 3$ and non-zero discriminant $D_F$. It was demonstrated by Bean \cite{bean94} that the quantities $A_F$ and $Q(F)$ remain invariant for all binary forms that are equivalent under $\operatorname{GL}_2(\mathbb Z)$ to $F$. Unlike $A_F$, the quantity $Q(F)$ also remains invariant for all binary forms that, up to multiplication by a complex number, are equivalent under $\operatorname{GL}_2(\mathbb R)$ to $F$. An important conjecture about $Q(F)$ was formulated by Bean \mbox{\cite[Conjecture 1]{bean95}}. Let
$$
F_n^*(x, y) = \prod\limits_{k = 1}^n\left(\sin\left(\frac{k\pi}{n}\right)x - \cos\left(\frac{k\pi}{n}\right)y\right).
$$

\begin{conj} \label{conj:bean}
The maximum value $M_n$ of $Q(F)$ over the forms $F$ with complex coefficients of degree $n$ with discriminant $D_F \neq 0$ is attained precisely when $F$ is a form which, up to multiplication by a complex number, is equivalent under $\operatorname{GL}_2(\mathbb R)$ to the form $F_n^*$. That is, $M_n = Q(F_n^*)$. Moreover, the limit of the sequence $\{M_n\}$ is $2\pi$.
\end{conj}

It was proved by Bean and Laugesen \cite[Theorem 1]{bean-laugesen} that
\begin{equation} \label{eq:QFn}
Q(F_n^*) = 2^{1 - 2/n}n^{1/(n - 1)}B\left(\frac{1}{2} - \frac{1}{n}, \frac{1}{2}\right)
\end{equation}
and that $\lim\limits_{n \rightarrow \infty} Q(F_n^*) = 2\pi$. Let $\nu_2(n)$ denote the $2$-adic order of $n$. In \cite{mosunov} it was proved by the author that the binary form
$$
S_n(x, y) = 2^{n - 1 - \nu_2(n)}F_n^*(x, y)
$$
has integer coefficients and that the greatest common divisor of its coefficients is equal to one.\footnote{Note that in \cite[Proposition 1]{mosunov} where the coefficients of $F_n^*$ and $S_n$ are determined there should be no $Y$ in front of summation symbols.} See Figure \ref{fig:S6} for the graph of $|S_6(x, y)| = 1$. Just like the graph of $|F_n^*(x, y)| = 1$, the graph of $|S_n(x, y)| = 1$ is invariant under rotation by any integer multiple of $\frac{\pi}{n}$. In \cite{mosunov} it was also demonstrated that
$$
A_{S_n} = 4^{\nu_2(n)/n}B\left(\frac{1}{2} - \frac{1}{n}, \frac{1}{2}\right),
$$
and as a consequence that $\lim\limits_{n \rightarrow \infty}A_{S_n} = \pi$. Since $Q(F_n^*)$ remains invariant under multiplication by a non-zero complex number, we see that $Q(S_n) = Q(2^{n - 1 - \nu_2(n)}F_n^*) = Q(F_n^*)$. In Section \ref{sec:bean} we will prove that $Q(T_n) \leq Q(S_n)$ and $Q(U_n) \leq Q(S_n)$ for every integer $n \geq 3$. We will also prove that
$$
\lim\limits_{n \rightarrow \infty} Q(\Psi_n) = \lim\limits_{n \rightarrow \infty} Q(\Pi_n) = \lim\limits_{n \rightarrow \infty} Q(T_n) = \lim\limits_{n \rightarrow \infty} Q(U_n) = \frac{16}{3}.
$$

\begin{figure}[t!]
\centering
\includegraphics[width=200pt]{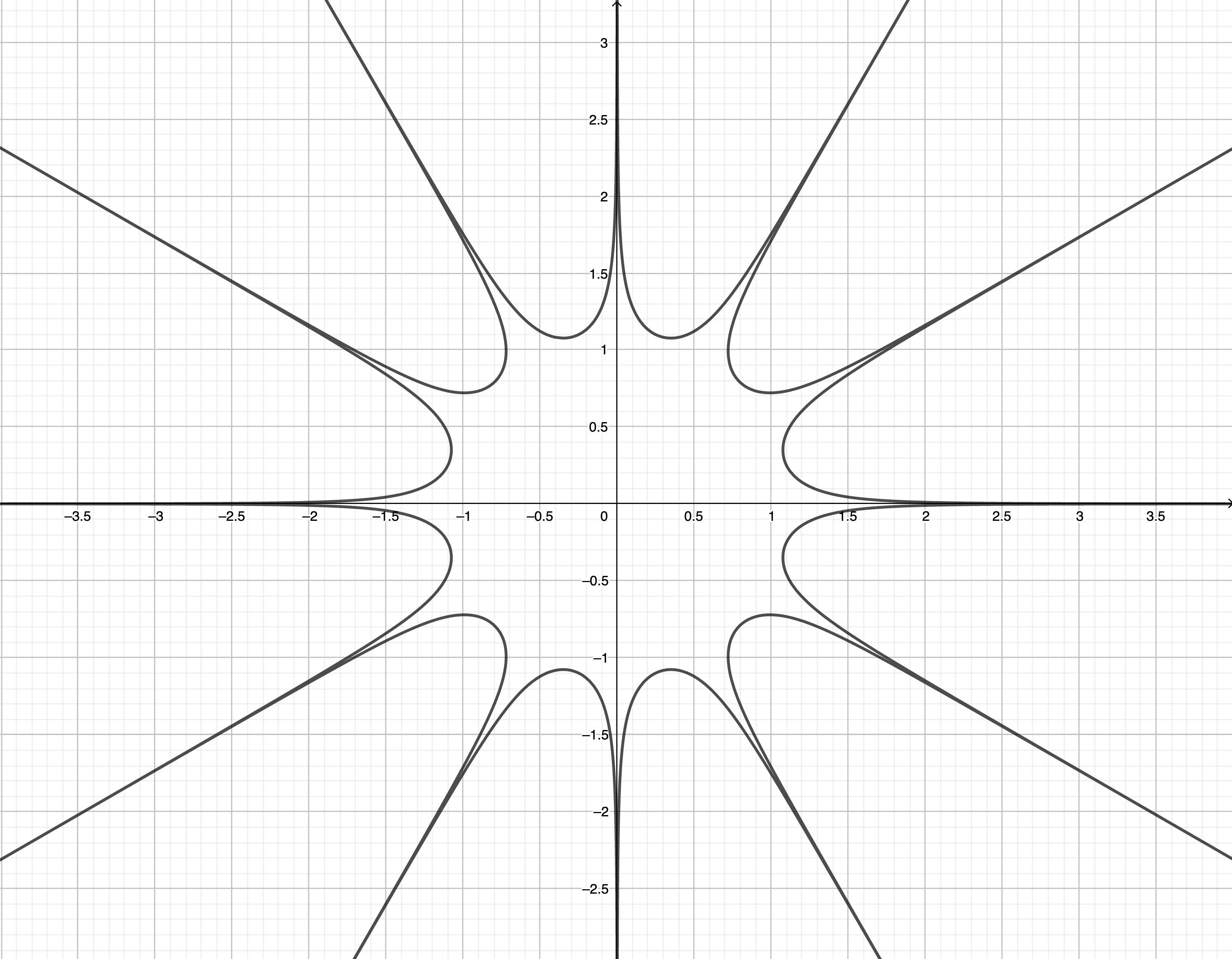}
\caption{Graph of $|S_6(x, y)| = 1$}
\label{fig:S6}
\end{figure}

\begin{table}[b]
\centering
\begin{tabular}{c | c | c | c | c | c}
$n$ & $Q(\Psi_n)$ & $Q(\Pi_n)$ & $Q(S_n)$ & $Q(T_n)$ & $Q(U_n)$\\
\hline
$3$ & $\infty$ & $\infty$ & $15.8997$ & $15.8997$ & $15.8997$\\
$4$ & $\infty$ & $\infty$ & $11.7726$ & $11.4798$ & $11.5438$\\
$5$ & $\infty$ & $10.8953$ & $10.3228$ & $9.85622$ & $9.94897$\\
$6$ & $\infty$ & $\infty$ & $9.55526$ & $8.97702$ & $9.08227$\\
$7$ & $15.8997$ & $8.5577$ & $9.06881$ & $8.41412$ & $8.52393$\\
$8$ & $\infty$ & $\infty$ & $8.72772$ & $8.01814$ & $8.12848$\\
$9$ & $15.8997$ & $9.00056$ & $8.47265$ & $7.72218$ & $7.83097$
\end{tabular}
\caption{Values of $Q(\Psi_n)$, $Q(\Pi_n)$, $Q(S_n)$, $Q(T_n)$ and $Q(U_n)$ for $n = 3, 4, \ldots, 9$.}
\label{tab:table2}
\end{table}

\section{Beta Function and Its Properties} \label{sec:beta}

For a complex number $z$ with positive real part, let
\begin{equation} \label{eq:gamma}
\Gamma(z) = \int\limits_0^\infty t^{z - 1}e^{-t}dt
\end{equation}
denote the \emph{gamma function}. This function satisfies the functional equation
\begin{equation} \label{eq:functional-equation}
\Gamma(z + 1) = z\Gamma(z).
\end{equation}
The gamma function is analytic, with simple pole at $x = 0$. Its residue at $x = 0$ is equal to $1$, i.e.,
\begin{equation} \label{eq:gamma-residue}
\operatorname{Res}(\Gamma(x), 0) = \lim\limits_{x \rightarrow 0}x\Gamma(x) = 1.
\end{equation}

The beta function (\ref{eq:beta-reg}) can be expressed in terms of the gamma function by means of the relation
\begin{equation} \label{eq:beta-gamma}
B(x, y) = \frac{\Gamma(x)\Gamma(y)}{\Gamma(x + y)}.
\end{equation}
If in (\ref{eq:beta-reg}) we make the change of variables $t = \sin^2 \theta$, then we obtain a trigonometric form of the beta function:
\begin{equation} \label{eq:beta}
B(x, y) = 2\int\limits_0^{\pi/2}(\sin \theta)^{2x - 1}(\cos\theta)^{2y - 1}d\theta.
\end{equation}
This representation of $B(x, y)$ is used in the derivation of the formula for $A_{F_n^*}$ \cite{mosunov}. In Lemmas \ref{lem:psi-area}, \ref{lem:chebyshev1-area} and \ref{lem:chebyshev2-area} we will use it to obtain upper and lower bounds on $A_{\Psi_n}$, $A_{T_n}$ and $A_{U_n}$.

Finally, we would like to remark that the values of beta function can be easily computed numerically, as this function is present in various mathematical software, such as Maple (\texttt{Beta(x, y)}), Mathematica  (\texttt{Beta(x, y)}) and MATLAB (\texttt{beta(Z, W)}).

\section{The Limit of $A_{F_n}$ for $F_n(x, y) = ax^n + by^n$} \label{sec:binomial-form}

In this section we prove the following result.

\begin{prop} \label{prop:binomial-forms}
Let $a$, $b$ and $n$ be non-zero integers with $n \geq 3$ and let $F_n(x, y) = ax^n + by^n$. Then $\lim\limits_{n \rightarrow \infty} A_{F_n} = 4$.
\end{prop}

\begin{proof}
In view of (\ref{eq:AFn}) and $\lim\limits_{n \rightarrow \infty}|ab|^{1/n} = 1$, the result will follow once we show that
$$
\lim\limits_{n \rightarrow \infty}\frac{1}{n}B\left(\frac{1}{n}, 1 - \frac{2}{n}\right) = 1 \quad \text{and} \quad
\lim\limits_{n \rightarrow \infty}\frac{1}{n}B\left(\frac{1}{n}, \frac{1}{n}\right) = 2.
$$
To see that this is the case, first note that the identities
$$
B\left(\frac{1}{n}, 1 - \frac{2}{n}\right) = \frac{\Gamma(1/n)\Gamma(1 - 2/n)}{\Gamma(1 - 1/n)} \quad \text{and} \quad B\left(\frac{1}{n}, \frac{1}{n}\right) = \frac{\Gamma(1/n)^2}{\Gamma(2/n)}
$$
follow directly from (\ref{eq:beta-gamma}). It follows from (\ref{eq:gamma-residue}) that
$$
\lim\limits_{n \rightarrow \infty} \frac{\Gamma(1/n)}{n} = 1 \quad \text{and} \quad \lim\limits_{n \rightarrow \infty} \frac{\Gamma(2/n)}{n} = \frac{1}{2}.
$$
Since $\Gamma(1) = 1$, we see that
\begin{align*}
\lim\limits_{n \rightarrow \infty}\frac{1}{n}B\left(\frac{1}{n}, 1 - \frac{2}{n}\right)
 = \lim\limits_{n \rightarrow \infty}\frac{\Gamma(1/n)\Gamma(1 - 2/n)}{n\Gamma(1 - 1/n)}
 = \lim\limits_{n \rightarrow \infty} \frac{\Gamma(1/n)}{n}
 = 1
\end{align*}
and
\begin{align*}
\lim\limits_{n \rightarrow \infty}\frac{1}{n}B\left(\frac{1}{n}, \frac{1}{n}\right)
 = \lim\limits_{n \rightarrow \infty}\frac{\Gamma(1/n)^2}{n\Gamma(2/n)}
 = \frac{\lim\limits_{n \rightarrow \infty}\left(\frac{\Gamma(1/n)}{n}\right)^2}{\lim\limits_{n \rightarrow \infty}\frac{\Gamma(2/n)}{n}}
 = 2.
\end{align*}
\end{proof}

\section{Upper and Lower Bounds on $A_{\Psi_n}$} \label{sec:Psi}

In this section we prove the following result.

\begin{lem} \label{lem:psi-area}
For any $n \in \mathbb N$ and $\alpha \in \mathbb R$ such that $\varphi(n) \geq 4$ and $2/\varphi(n) < \alpha < 1$,
\small
\begin{align*}
\int\limits_{-\infty}^{+\infty}\frac{dx}{|\Psi_n(x)|^\alpha} & > \left(4 + \frac{16}{\varphi(n)^2\alpha^2 - 4}\right)\exp\left(-\frac{d(n) \log n}{2}\alpha\right)\\
\int\limits_{-\infty}^{+\infty}\frac{dx}{|\Psi_n(x)|^\alpha} & < 2^{\alpha}\exp\left(\frac{d(n)^2\log n}{2}\alpha\right), \times\\
& \times \left[4 + 2^{1-\alpha}B\left(\frac{1 - \alpha}{2}, \frac{1}{2}\right) - 2\pi + \frac{2}{n}B\left(\frac{\varphi(n)\alpha - 2}{2n}, 1 - \alpha\right) - \frac{2}{n}B\left(\frac{\varphi(n)\alpha + 2}{2n}, 1 - \alpha\right)\right].
\end{align*}
\normalsize
\end{lem}

Let us see why Theorem \ref{thm:main3} directly follows from Lemma \ref{lem:psi-area}.

\begin{proof}[Proof of Theorem \ref{thm:main3}]
It follows from (\ref{eq:Psin-formula}) that the degree of $\Psi_n$ is equal to $\varphi(n)/2$. Since the area $A_{\Psi_n}$ bounded by the curve $|\Psi_n(x, y)| = 1$ can be computed via (\ref{eq:AF-formula}), the inequalities (\ref{eq:APsin-bounds}) follow from Lemma \ref{lem:psi-area} once we take $\alpha = 4/\varphi(n)$.

The fact that $\lim\limits_{n \rightarrow \infty} A_{\Psi_n} = \frac{16}{3}$ is a consequence of the following:
$$
\lim\limits_{n \rightarrow \infty} B\left(\frac{1}{2}, \frac{1}{2} - \frac{2}{\varphi(n)}\right) = B\left(\frac{1}{2}, \frac{1}{2}\right) = \frac{\Gamma(1/2)^2}{\Gamma(1)} = \pi,
$$
$$
\lim\limits_{n \rightarrow \infty} \frac{2}{n}B\left(\frac{1}{n}, 1 - \frac{4}{\varphi(n)}\right) = \lim\limits_{n \rightarrow \infty} \frac{2}{n}\Gamma\left(\frac{1}{n}\right) = 2,
$$
$$
\lim\limits_{n \rightarrow \infty} \frac{2}{n}B\left(\frac{3}{n}, 1 - \frac{4}{\varphi(n)}\right) = \lim\limits_{n \rightarrow \infty} \frac{2}{n}\Gamma\left(\frac{3}{n}\right) = \frac{2}{3}.
$$
Further, we have
$$
\lim\limits_{n \rightarrow \infty} \frac{2d(n)\log n}{\varphi(n)} = 0 \quad \text{and} \quad \lim\limits_{n \rightarrow \infty} \frac{2d(n)^2\log n}{\varphi(n)} = 0.
$$
These two limits follow from the fact that, for every integer $n > 2$,
\begin{equation} \label{eq:d}
d(n) \leq n^{\frac{1.067}{\log\log n}}
\end{equation}
and
\begin{equation} \label{eq:phi-lower-bound}
\varphi(n) > \frac{n}{e^\gamma \log\log n + \frac{3}{\log \log n}}.
\end{equation}
Here $\gamma \approx 0.5772$ denotes the Euler-Mascheroni constant. The upper bound on $d(n)$ was derived by Nicolas and Robin \cite{nicolas-robin}, while the lower bound on $\varphi(n)$ can be found in the paper of Rosser and Schoenfeld \cite[Theorem 15]{rosser-schoenfeld}.
\end{proof}

Before proving Lemma \ref{lem:psi-area}, we need to establish three lemmas.

\begin{lem} \label{lem:psi}
Let $n$ be an integer such that $n \geq 3$. Then, for every \mbox{real number $\theta$},
\begin{equation} \label{eq:psi-cos}
\Psi_n(\pm 2\cos \theta) = (\pm 1)^{\frac{\varphi(n)}{2}}e^{-\frac{\varphi(n)}{2}\theta i}\Phi_n\left(\pm e^{\theta i}\right),
\end{equation}
\begin{equation} \label{eq:psi-cosh}
\Psi_n(\pm 2\cosh \theta) = (\pm 1)^{\frac{\varphi(n)}{2}}e^{\frac{\varphi(n)}{2}\theta}\Phi_n\left(\pm e^{-\theta}\right). 
\end{equation}
\end{lem}

\begin{proof}
According to Lehmer \cite{lehmer},
$$
\Psi_n(x + x^{-1}) = x^{-\varphi(n)/2}\Phi_n(x).
$$
Putting $x = e^{\theta i}$, we obtain the identity (\ref{eq:psi-cos}), where the plus signs are chosen everywhere. Replacing $\theta$ with $\pi + \theta$, $\theta i$ and $\pi + \theta i$ we obtain the identity (\ref{eq:psi-cos}) where the minus signs are chosen everywhere, as well as the two identities in (\ref{eq:psi-cosh}).
\end{proof}

For a polynomial $f(x)$, let $L(f)$ denote the sum of the absolute values of coefficients of $f(x)$. The following lemma can be found in \cite[Lemme 4.1]{fouvry-waldschmidt}.

\begin{lem} \label{lem:L-bound}
For every integer $n > 1$, $L(\Phi_n) \leq n^{\frac{d(n)}{2}}$.
\end{lem}

\begin{lem} \label{lem:prod-L-bound}
For every positive integer $n$,
$$
\prod\limits_{\substack{m \mid n\\m \neq 1}}L(\Phi_m) \leq \exp\left(\frac{d(n)^2\log n}{2}\right).
$$
\end{lem}

\begin{proof}
By Lemma \ref{lem:L-bound},
\begin{align*}
\prod\limits_{\substack{m \mid n\\m \neq 1}}L(\Phi_m)
& \leq \prod\limits_{m \mid n}m^{\frac{d(m)}{2}}\\
& = \exp\left(\frac{1}{2}\sum\limits_{m \mid n}d(m)\log m\right)\\
& \leq \exp\left(\frac{d(n)\log n}{2}\sum\limits_{m \mid n}1\right)\\
& = \exp\left(\frac{d(n)^2\log n}{2}\right).
\end{align*}
\end{proof}

\begin{proof}[Proof of Lemma \ref{lem:psi-area}]
Note that
\small
\begin{align*}
\int\limits_{-\infty}^{+\infty}\frac{dx}{|\Psi_n(x)|^{\alpha}}
& = \int\limits_{-2}^2\frac{dx}{|\Psi_n(x)|^{\alpha}} + \int\limits_{2}^{+\infty}\frac{dx}{|\Psi_n(x)|^{\alpha}} + \int\limits_{-\infty}^{-2}\frac{dx}{|\Psi_n(x)|^{\alpha}}\\
& = -\int\limits_{0}^\pi\frac{d(2\cos\theta)}{|\Psi_n(2\cos \theta)|^{\alpha}} + \int\limits_{0}^{+\infty}\frac{d(2\cosh\theta)}{|\Psi_n(2\cosh\theta)|^{\alpha}} + \int\limits_{-\infty}^{0}\frac{d(-2\cosh\theta)}{|\Psi_n(-2\cosh\theta)|^{\alpha}}\\
& = 2\int\limits_{0}^{\pi}\frac{\sin\theta d\theta}{|\Psi_n(2\cos \theta)|^{\alpha}} + 2\int\limits_{0}^{+\infty}\frac{\sinh\theta d\theta}{|\Psi_n(2\cosh\theta)|^{\alpha}} + 2\int\limits_{0}^{+\infty}\frac{\sinh\theta d\theta}{|\Psi_n(-2\cosh\theta)|^{\alpha}}\\
& = 2\int\limits_{0}^{\pi}\frac{\sin\theta d\theta}{|\Phi_n(e^{\theta i})|^{\alpha}} + 2\int\limits_{0}^{+\infty}\frac{\sinh\theta d\theta}{e^{(\varphi(n)/2)\alpha\theta}\Phi_n(e^{-\theta})^{\alpha}} + 2\int\limits_{0}^{+\infty}\frac{\sinh\theta d\theta}{e^{(\varphi(n)/2)\alpha\theta}\Phi_n(-e^{-\theta})^{\alpha}},
\end{align*}
\normalsize
where the last equality follows from Lemma \ref{lem:psi}. Let
\small
$$
C_1 = 2\int\limits_{0}^{\pi}\frac{\sin\theta d\theta}{|\Phi_n(e^{\theta i})|^{\alpha}}, \quad C_2 = 2\int\limits_{0}^{+\infty}\frac{\sinh\theta d\theta}{e^{(\varphi(n)/2)\alpha\theta}\Phi_n(e^{-\theta})^{\alpha}}, \quad C_3 = 2\int\limits_{0}^{+\infty}\frac{\sinh\theta d\theta}{e^{(\varphi(n)/2)\alpha\theta}\Phi_n(-e^{-\theta})^{\alpha}}.
$$
\normalsize
We will now derive upper and lower bounds on $C_1$, $C_2$ and $C_3$.

We begin with the derivation of a lower bound. By Lemma \ref{lem:L-bound},
$$
|\Phi_n(z)| \leq L(\Phi_n)\max\{1, |z|\}^{\varphi(n)} \leq n^{\frac{d(n)}{2}}
$$
for all $z \in \mathbb C$ such that $|z| \leq 1$. In view of this,
\begin{align*}
C_1
& = 2\int\limits_{0}^{\pi}\frac{\sin\theta d\theta}{|\Phi_n(e^{\theta i})|^{\alpha}}\\
& \geq 2n^{-\frac{d(n)}{2}\alpha}\int\limits_{0}^{\pi}\sin\theta d\theta\\
& = 4\exp\left(-\frac{d(n) \log n}{2}\alpha\right)
\end{align*}
and
\begin{align*}
C_2
& = 2\int\limits_{0}^{+\infty}\frac{\sinh\theta d\theta}{e^{(\varphi(n)/2)\alpha\theta}\Phi_n\left(e^{-\theta}\right)^{\alpha}}\\
& \geq 2n^{-\frac{d(n)}{2}\alpha}\int\limits_0^{+\infty}e^{-\frac{\varphi(n)}{2}\alpha\theta}\sinh\theta d\theta\\
& = 2n^{-\frac{d(n)}{2}\alpha}\frac{1}{\left(\frac{\varphi(n)}{2}\alpha\right)^2 - 1}\\
& =  \frac{8}{\varphi(n)^2\alpha^2 - 4}\exp\left(-\frac{d(n) \log n}{2}\alpha\right).
\end{align*}
Similarly,
$$
C_3 \geq \frac{8}{\varphi(n)^2\alpha^2 - 4}\exp\left(-\frac{d(n) \log n}{2}\alpha\right).
$$
We conclude that
$$
\int\limits_{-\infty}^{+\infty}\frac{dx}{|\Psi_n(x)|^{\alpha}} = C_1 + C_2 + C_3 \geq \left(4 + \frac{16}{\varphi(n)^2\alpha^2 - 4}\right)\exp\left(-\frac{d(n) \log n}{2}\alpha\right).
$$

To derive an upper bound, we will make use of the identity
$$
z^n - 1 = \prod\limits_{m \mid n}\Phi_m(z).
$$
Notice that
$$
\Phi_n(z) = (z^n - 1) \prod\limits_{\substack{m \mid n\\ m \neq n}}\Phi_m(z)^{-1}.
$$
Thus for every $z \in \mathbb C$ such that $|z| < 1$ we have
\begin{align*}
|\Phi_n(z)|^{-1}
& = \frac{1}{|z^n - 1|}\prod\limits_{\substack{m \mid n\\ m \neq n}}|\Phi_m(z)|\\
& = \frac{|z - 1|}{|z^n - 1|}\prod\limits_{\substack{m \mid n\\ 1 < m < n}}|\Phi_m(z)|\\
& \leq \frac{|z - 1|}{|z^n - 1|}\prod\limits_{\substack{m \mid n\\m \neq 1}}L(\Phi_m)\\
& < \frac{2}{|z^n - 1|}\exp\left(\frac{d(n)^2\log n}{2}\right),
\end{align*}
where the last inequality follows from Lemma \ref{lem:prod-L-bound}. In view of this,
\begin{align*}
C_1
& = 2\int\limits_{0}^{\pi}\frac{\sin\theta d\theta}{|\Phi_n(e^{\theta i})|^{\alpha}}\\
& \leq 2^{1 + \alpha}\exp\left(\frac{d(n)^2\log n}{2}\alpha\right)\int\limits_0^{\pi}\frac{\sin\theta d\theta}{|e^{n\theta i} - 1|^{\alpha}}\\
& = 2^{1 + \alpha}\exp\left(\frac{d(n)^2\log n}{2}\alpha\right)\left(\int\limits_0^{\pi}\sin\theta d\theta + \int\limits_0^{\pi}\sin\theta\left(|e^{n\theta i} - 1|^{-\alpha} - 1\right) d\theta\right)\\
&\leq 2^{1 + \alpha}\exp\left(\frac{d(n)^2\log n}{2}\alpha\right)\left(2 + \int\limits_0^{\pi}\left(\frac{1}{|e^{n\theta i} - 1|^{\alpha}} - 1\right) d\theta\right)\\
& = 2^{1 + \alpha}\exp\left(\frac{d(n)^2\log n}{2}\alpha\right)\left(2 + \int\limits_0^{\pi}\frac{d\theta}{|e^{n\theta i} - 1|^{\alpha}} - \pi\right).
\end{align*}
It remains to estimate the integral $\int_0^{\pi}\frac{d\theta}{|e^{n\theta i} - 1|^{\alpha}}$. Note that
$$
|e^{n\theta i} - 1|^2 = \left(\cos(n\theta) - 1\right)^2 + \sin^2(n\theta) = 2 - 2\cos(n\theta) = 4\sin^2\left(\frac{n\theta}{2}\right).
$$
Therefore,
\begin{align*}
\int_0^{\pi}\frac{d\theta}{|e^{n\theta i} - 1|^{\alpha}}
& = 2^{-\alpha}\int\limits_0^{\pi}\left|\sin\left(\frac{n\theta}{2}\right)\right|^{-\alpha}d\theta\\
& = \frac{2^{1-\alpha}}{n}\int\limits_0^{n\pi/2}\left|\sin \theta\right|^{-\alpha}d\theta\\
& = 2^{-\alpha}\cdot 2\int\limits_0^{\pi/2}\left(\sin \theta\right)^{-\alpha}d\theta\\
& = 2^{-\alpha}B\left(\frac{1 - \alpha}{2}, \frac{1}{2}\right).
\end{align*}
We conclude that
\begin{align*}
C_1
& = 2\int\limits_{0}^{\pi}\frac{\sin\theta d\theta}{|\Phi_n(e^{\theta i})|^{\alpha}}\\
& \leq 2^{\alpha}\exp\left(\frac{d(n)^2\log n}{2}\alpha\right)\left(4 + 2^{1-\alpha}B\left(\frac{1 - \alpha}{2}, \frac{1}{2}\right) - 2\pi\right).
\end{align*}

Next, we derive an upper bound on $C_2$:
\begin{align*}
C_2
& = 2\int\limits_{0}^{+\infty}\frac{\sinh\theta d\theta}{e^{(\varphi(n)/2)\alpha\theta}\Phi_n(e^{-\theta})^{\alpha}}\\
& \leq 2^{1 + \alpha}\exp\left(\frac{d(n)^2\log n}{2}\alpha\right)\int\limits_0^{+\infty}\frac{\sinh\theta d\theta}{e^{(\varphi(n)/2)\alpha\theta}(1 - e^{-n\theta})^{\alpha}}\\
& = 2^{1 + \alpha}\exp\left(\frac{d(n)^2\log n}{2}\alpha\right)\frac{1}{2n}\int\limits_0^1\left(t^{\frac{\varphi(n)\alpha - 2}{2n} - 1} - t^{\frac{\varphi(n)\alpha + 2}{2n} - 1}\right)(1 - t)^{-\alpha} dt\\
& = 2^{\alpha}\exp\left(\frac{d(n)^2\log n}{2}\alpha\right)\left(\frac{1}{n}B\left(\frac{\varphi(n)\alpha - 2}{2n}, 1 - \alpha\right) - \frac{1}{n}B\left(\frac{\varphi(n)\alpha + 2}{2n}, 1 - \alpha\right)\right),
\end{align*}
where the second-to-last equality follows from the change of variables $t = e^{-n\theta}$.

Finally, we derive an upper bound on $C_3$. We consider two cases.

\emph{Case 1.} Suppose that $n$ is even. Then
\begin{align*}
C_3 & = 2\int\limits_{0}^{+\infty}\frac{\sinh\theta d\theta}{e^{(\varphi(n)/2)\alpha\theta}\Phi_n(-e^{-\theta})^{\alpha}}\\
& \leq 2^{1 + \alpha}\exp\left(\frac{d(n)^2\log n}{2}\alpha\right)\int\limits_0^{+\infty}\frac{\sinh\theta d\theta}{e^{(\varphi(n)/2)\alpha\theta}\left(1 - (-e^{-\theta})^n\right)^{\alpha}}\\
 & = 2^{1 + \alpha}\exp\left(\frac{d(n)^2\log n}{2}\alpha\right)\int\limits_0^{+\infty}\frac{\sinh\theta d\theta}{e^{(\varphi(n)/2)\alpha\theta}\left(1 - e^{-n\theta}\right)^{\alpha}}\\
& = 2^{\alpha}\exp\left(\frac{d(n)^2\log n}{2}\alpha\right)\left(\frac{1}{n}B\left(\frac{\varphi(n)\alpha - 2}{2n}, 1 - \alpha\right) - \frac{1}{n}B\left(\frac{\varphi(n)\alpha + 2}{2n}, 1 - \alpha\right)\right),
\end{align*}
where the last equality was derived from the estimate for $C_2$.

\emph{Case 2.} Suppose that $n$ is odd. Then for any real number $x$ we have $\Phi_n(-x) = \Phi_{2n}(x)$, so
\begin{align*}
C_3 & = 2\int\limits_{0}^{+\infty}\frac{\sinh\theta d\theta}{e^{(\varphi(n)/2)\alpha\theta}\Phi_n(-e^{-\theta})^{\alpha}}\\
& = 2\int\limits_{0}^{+\infty}\frac{\sinh\theta d\theta}{e^{(\varphi(n)/2)\alpha\theta}\Phi_{2n}(e^{-\theta})^{\alpha}}\\
& \leq 2^{1 + \alpha}\exp\left(\frac{d(n)^2\log n}{2}\alpha\right) \int\limits_0^{+\infty}\frac{\sinh\theta d\theta}{e^{(\varphi(n)/2)\alpha\theta}\left(1 - e^{-2n\theta}\right)^{\alpha}}\\
& \leq 2^{1 + \alpha}\exp\left(\frac{d(n)^2\log n}{2}\alpha\right) \int\limits_0^{+\infty}\frac{\sinh\theta d\theta}{e^{(\varphi(n)/2)\alpha\theta}\left(1 - e^{-n\theta}\right)^{\alpha}}\\
& = 2^{\alpha}\exp\left(\frac{d(n)^2\log n}{2}\alpha\right)\left(\frac{1}{n}B\left(\frac{\varphi(n)\alpha - 2}{2n}, 1 - \alpha\right) - \frac{1}{n}B\left(\frac{\varphi(n)\alpha + 2}{2n}, 1 - \alpha\right)\right),
\end{align*}
where the last equality was derived from the estimate for $C_2$.

We conclude that
\small
\begin{align*}
\int\limits_{-\infty}^{+\infty}\frac{dx}{|\Psi_n(x)|^{\alpha}}
& = C_1 + C_2 + C_3\\
& \leq 2^{\alpha}\exp\left(\frac{d(n)^2\log n}{2}\alpha\right) \times\\
& \times \left(4 + 2^{1-\alpha}B\left(\frac{1 - \alpha}{2}, \frac{1}{2}\right) - 2\pi + \frac{2}{n}B\left(\frac{\varphi(n)\alpha - 2}{2n}, 1 - \alpha\right) - \frac{2}{n}B\left(\frac{\varphi(n)\alpha + 2}{2n}, 1 - \alpha\right)\right).
\end{align*}
\normalsize
\end{proof}

\section{Proof of Corollary \ref{cor:main3}} \label{sec:pi}

In view of the identity $\sin\left(\frac{2\pi}{n}\right) = \cos\left(\frac{2\pi(n - 4)}{4n}\right)$, we can follow the argument of Niven \cite[III.4]{niven} and consider the following five cases.
\begin{enumerate}[(1)]
\item If $n$ is odd, then the fraction $\frac{n - 4}{4n}$ is in lowest terms. Thus $2\cos\left(\frac{2\pi(n - 4)}{4n}\right)$ is conjugate to $2\cos\left(\frac{2\pi}{4n}\right)$ and $c(n) = 4n$.

\item If $n \equiv 2 \pmod 4$, then $\frac{n - 4}{4n}$ reduces to a fraction with denominator $2n$. Thus $2\cos\left(\frac{2\pi(n - 4)}{4n}\right)$ is conjugate to $2\cos\left(\frac{2\pi}{2n}\right)$ and $c(n) = 2n$.

\item If $n \equiv 0 \pmod 8$, then $\frac{n - 4}{4n}$ reduces to a fraction with denominator $n$. Thus $2\cos\left(\frac{2\pi(n - 4)}{4n}\right)$ is conjugate to $2\cos\left(\frac{2\pi}{n}\right)$ and $c(n) = n$.

\item If $n \equiv 12 \pmod{16}$, then $\frac{n-4}{4n}$ reduces to a fraction with denominator $n/2$. Thus $2\cos\left(\frac{2\pi(n - 4)}{4n}\right)$ is conjugate to $2\cos\left(\frac{2\pi}{n/2}\right)$ and $c(n) = n/2$.

\item If $n \equiv 4 \pmod{16}$, then $\frac{n - 4}{4n}$ reduces to a fraction with denominator $n/4$. Thus $2\cos\left(\frac{2\pi(n - 4)}{4n}\right)$ is conjugate to $2\cos\left(\frac{2\pi}{n/4}\right)$ and $c(n) = n/4$.
\end{enumerate}
A combination of this result with Theorem \ref{thm:main3} yields \mbox{$\lim\limits_{n \rightarrow \infty} A_{\Pi_n} = \lim\limits_{n \rightarrow \infty} A_{\Psi_n} = \frac{16}{3}$}.

\section{Upper and Lower Bounds on $A_{T_n}$} \label{sec:Tn}

In this section we prove the following result.

\begin{lem} \label{lem:chebyshev1-area}
For any $n \in \mathbb N$ and $\alpha \in \mathbb R$ such that $n \geq 3$ and $2/n \leq \alpha < 1$,
$$
2 + \frac{2}{n^2\alpha^2 - 1} < \int_{-\infty}^{+\infty}\frac{dx}{|T_n(x)|^\alpha} < 2 + \frac{2^{1 + \alpha}}{n^2\alpha^2 - 1} + B\left(\frac{1 - \alpha}{2}, \frac{1}{2}\right) - \pi.
$$
\end{lem}

Let us see why Theorem \ref{thm:main1} directly follows from Lemma \ref{lem:chebyshev1-area}.

\begin{proof}[Proof of Theorem \ref{thm:main1}]
Since for an arbitrary binary form $F$ of positive degree $n$ the area bounded by the curve $|F(x, y)| = 1$ can be computed via (\ref{eq:AF-formula}), the inequalities (\ref{eq:ATn-bounds}) follow from Lemma \ref{lem:chebyshev1-area} if we take $\alpha = 2/n$.

To see that $\lim\limits_{n \rightarrow \infty} A_{T_n} = \frac{8}{3}$, it is sufficient to demonstrate that the limit of $B\left(\frac{1}{2} - \frac{1}{n}, \frac{1}{2}\right)$ as $n$ approaches infinity is equal to $\pi$. Recall that the gamma function (\ref{eq:gamma}) satisfies $\Gamma(1) = 1$ and $\Gamma(1/2) = \sqrt{\pi}$. In view of (\ref{eq:beta-gamma}), we see that
$$
\lim\limits_{n \rightarrow \infty}B\left(\frac{1}{2} - \frac{1}{n}, \frac{1}{2}\right) = B\left(\frac{1}{2}, \frac{1}{2}\right) = \frac{\Gamma(1/2)^2}{\Gamma(1)} = \pi.
$$
\end{proof}

We will now turn our attention to the proof of Lemma \ref{lem:chebyshev1-area}.

\begin{proof}[Proof of Lemma \ref{lem:chebyshev1-area}]
Since the function $|T_n(x)|$ is even, we find that
$$
\int_{-\infty}^{+\infty}\frac{dx}{|T_n(x)|^\alpha} = C_1 + C_2,
$$
where
$$
C_1 = 2\int\limits_{1}^{+\infty}\frac{dx}{T_n(x)^\alpha}, \,\,\, C_2 = 2\int\limits_{0}^{1}\frac{dx}{|T_n(x)|^\alpha},
$$
We will now derive upper and lower bounds on $C_1$ and $C_2$.

First, we determine upper and lower bounds on $C_1$. Recall that
$$
T_n(x) = \frac{(x + \sqrt{x^2 - 1})^n + (x - \sqrt{x^2 - 1})^n}{2}
$$
for all $x$ such that $|x| \geq 1$. Hence
\begin{equation} \label{eq:Tn-bound}
\frac{(x + \sqrt{x^2 - 1})^n}{2} < T_n(x) < (x + \sqrt{x^2 - 1})^n
\end{equation}
for all $x > 1$. Using the lower bound in (\ref{eq:Tn-bound}), as well as the change of variables $x = \csc \theta$, we find that
\begin{align*}
C_1
& = 2\int\limits_{1}^{+\infty}\frac{dx}{T_n(x)^\alpha}\\
& < 2^{\alpha + 1}\int\limits_1^{+\infty}\frac{dx}{(x + \sqrt{x^2 - 1})^{n\alpha}}\\
& = 2^{\alpha + 1}\int\limits_0^{\pi/2}\frac{(\cos \theta)(\sin \theta)^{n \alpha - 2}d\theta}{(1 + \cos \theta)^{n\alpha}}\\
& = 2^{\alpha + 1}\left.\frac{(\sin \theta)^{n\alpha - 1}(n\alpha\cos \theta + 1)}{(n^2\alpha^2 - 1)(1 + \cos \theta)^{n\alpha}}\right|_{0}^{\pi/2}\\
& = \frac{2^{\alpha + 1}}{n^2\alpha^2 - 1}
\end{align*}
%
%
A lower bound on $C_1$ can be obtained analogously. As a result, we get
$$
\frac{2}{n^2\alpha^2 - 1} < C_1 < \frac{2^{1 + \alpha}}{n^2\alpha^2 - 1}.
$$

Next, we determine upper and lower bounds on $C_2$. We apply the substitution $x = \cos(\theta/n)$ and make use of the identity $\cos \theta = T_n(\cos(\theta/n))$ as follows:\footnote{The author is grateful to Fedor Petrov for recognizing the third equality. See \url{https://mathoverflow.net/questions/345820/an-integral-of-sinx-cosnx-2-n-from-pi-to-pi}.}
\begin{align*}
C_2
& = 2\int\limits_0^1\frac{dx}{|T_n(x)|^\alpha}\\
& = \frac{2}{n}\int\limits_0^{\pi n/2}\frac{\sin(\theta/n)d\theta}{|\cos \theta|^\alpha}\\
& = \frac{2}{n}\int\limits_0^{\pi n/2}\sin(\theta/n)d\theta + \varepsilon_n,
\end{align*}
where
$$
\varepsilon_n = \frac{2}{n}\int\limits_0^{\pi n/2}\sin(\theta/n)\left(|\cos \theta|^{-\alpha} - 1\right)d\theta.
$$
Note that $\varepsilon_n > 0$. Since
$$
\frac{2}{n}\int\limits_0^{\pi n/2}\sin(\theta/n)d\theta = 2\int\limits_0^{\pi/2}\sin\theta d\theta = 2,
$$
we have
$$
C_2 = \frac{2}{n}\int\limits_0^{\pi n/2}\sin(\theta/n)d\theta + \varepsilon_n > 2.
$$

It remains to find an upper bound on $C_2$. Note that
\begin{align*}
\varepsilon_n & = \frac{2}{n}\int\limits_0^{\pi n/2}\sin(\theta/n)\left(|\cos \theta|^{-\alpha} - 1\right)d\theta\\
& \leq \frac{2}{n}\int\limits_0^{\pi n/2}|\cos \theta|^{-\alpha}d\theta - \pi\\
& = 2\int\limits_0^{\pi/2}(\cos \theta)^{-\alpha}d\theta - \pi\\
& = B\left(\frac{1 - \alpha}{2}, \frac{1}{2}\right) - \pi,
\end{align*}
where the last equality follows from (\ref{eq:beta}). We conclude that
$$
2 < C_2 \leq 2 + B\left(\frac{1 - \alpha}{2}, \frac{1}{2}\right) - \pi.
$$
\end{proof}

\section{Upper and Lower Bounds on $A_{U_n}$} \label{sec:Un}

In this section we prove the following result.

\begin{lem} \label{lem:chebyshev2-area}
For any $n \in \mathbb N$ and $\alpha \in \mathbb R$ such that $n \geq 3$ and $2/n \leq \alpha < 1$,
\begin{align*}
\int\limits_{-\infty}^{+\infty}\frac{dx}{|U_n(x)|^\alpha} & > B\left(\frac{2 + \alpha}{2}, \frac{1}{2}\right) + \frac{2}{n^2\alpha^2 - 1}(n + 1)^{-\alpha}\\
\int\limits_{-\infty}^{+\infty}\frac{dx}{|U_n(x)|^\alpha} & < B\left(\frac{2 + \alpha}{2}, \frac{1}{2}\right) + \frac{2}{n^2\alpha^2 - 1} + B\left(\frac{1 - \alpha}{2}, \frac{1}{2}\right) - \pi\notag
\end{align*}
\end{lem}

Let us see why Theorem \ref{thm:main2} directly follows from Lemma \ref{lem:chebyshev2-area}.

\begin{proof}[Proof of Theorem \ref{thm:main2}]
Since for an arbitrary binary form $F$ of positive degree $n$ the area bounded by the curve $|F(x, y)| = 1$ can be computed via (\ref{eq:AF-formula}), the inequalities (\ref{eq:AUn-bounds}) follow from Lemma \ref{lem:chebyshev2-area} if we take $\alpha = 2/n$.

To see that $\lim\limits_{n \rightarrow \infty} A_{U_n} = \frac{8}{3}$, it is sufficient to demonstrate that the limits of $B\left(\frac{1}{2} - \frac{1}{n}, \frac{1}{2}\right)$ and $B\left(1 + \frac{1}{n}, \frac{1}{2}\right)$ as $n$ approaches infinity are equal to $\pi$ and $2$, respectively. Recall that the gamma function (\ref{eq:gamma}) satisfies $\Gamma(1) = 1$ and $\Gamma(1/2) = \sqrt{\pi}$. In view of (\ref{eq:beta-gamma}), we see that
$$
\lim\limits_{n \rightarrow \infty}B\left(\frac{1}{2} - \frac{1}{n}, \frac{1}{2}\right) = B\left(\frac{1}{2}, \frac{1}{2}\right) = \frac{\Gamma(1/2)^2}{\Gamma(1)} = \pi
$$
and
$$
\lim\limits_{n \rightarrow \infty}B\left(1 + \frac{1}{n}, \frac{1}{2}\right) = B\left(1, \frac{1}{2}\right) = \frac{\Gamma(1)\Gamma(1/2)}{\Gamma(3/2)} = \frac{\Gamma(1)\Gamma(1/2)}{(1/2)\Gamma(1/2)} = 2,
$$
where the second-to-last equality follows from the functional equation (\ref{eq:functional-equation}).
\end{proof}

We will now turn our attention to the proof of Lemma \ref{lem:chebyshev2-area}.

\begin{proof}[Proof of Lemma \ref{lem:chebyshev2-area}]
Since the function $|U_n(x)|$ is even, we find that
$$
\int_{-\infty}^{+\infty}\frac{dx}{|U_n(x)|^\alpha} = C_1 + C_2,
$$
where
$$
C_1 = 2\int\limits_{1}^{+\infty}\frac{dx}{U_n(x)^{\alpha}}, \,\,\, C_2 = 2\int\limits_{0}^{1}\frac{dx}{|U_n(x)|^{\alpha}}.
$$
We will now derive upper and lower bounds on $C_1$ and $C_2$.

Put $C_1 = 2\int\limits_{1}^{+\infty}\frac{dx}{U_n(x)^{\alpha}}, \,\,\, C_2 = 2\int\limits_{0}^{1}\frac{dx}{|U_n(x)|^{\alpha}}$, so that $\int\limits_{-\infty}^{+\infty}\frac{dx}{|U_n(x)|^\alpha} = C_1 + C_2$.

First, we determine upper and lower bounds on $C_1$. Recall that
$$
U_n(x) = \frac{\left(x + \sqrt{x^2 - 1}\right)^{n + 1} - \left(x - \sqrt{x^2 - 1}\right)^{n + 1}}{2\sqrt{x^2 - 1}}
$$
for all $x$ such that $|x| \geq 1$.
%
Thus for every $\theta \in [0, \pi/2]$ we have
\begin{align*}
(\sin \theta)^nU_n(\csc \theta)
& = 
\frac{(1 + \cos \theta)^{n + 1} - (1 - \cos \theta)^{n + 1}}{2\cos \theta}\\
& =
\frac{\left(2\cos^2\frac{\theta}{2}\right)^{n + 1} - \left(2\sin^2\frac{\theta}{2}\right)^{n + 1}}{2\cos \theta}\\
& = 2^n\sum\limits_{k = 0}^n\left(\sin^2\frac{\theta}{2}\right)^k\left(\cos^2\frac{\theta}{2}\right)^{n - k}.
\end{align*}
Since
$$
\left(\cos^2\frac{\theta}{2}\right)^n < \sum\limits_{k = 0}^n\left(\sin^2\frac{\theta}{2}\right)^k\left(\cos^2\frac{\theta}{2}\right)^{n - k} < (n + 1)\left(\cos^2\frac{\theta}{2}\right)^n
$$
for every $\theta \in (0, \pi/2]$, we find that
\begin{equation} \label{eq:Un-bound}
\left(1 + \cos \theta\right)^n < (\sin \theta)^nU_n(\csc\theta) < (n + 1)\left(1 + \cos \theta\right)^n
\end{equation}
for every $\theta \in (0, \pi/2]$.


 Using the change of variables $x = \csc \theta$, as well as the lower bound in (\ref{eq:Un-bound}), we find that
\begin{align*}
C_1
& = 2\int\limits_1^{+\infty}\frac{dx}{U_n(x)^\alpha}\\
& = 2\int\limits_0^{\pi/2}\frac{\cos\theta d\theta}{(\sin \theta)^2U_n(\csc\theta)^\alpha}\\
& < 2\int\limits_0^{\pi/2}\frac{\cos\theta (\sin \theta)^{n\alpha - 2} d\theta}{(1 + \cos\theta)^{n\alpha}}\\
& = 2\left.\frac{(\sin \theta)^{n\alpha - 1}(n\alpha\cos \theta + 1)}{(n^2\alpha^2 - 1)(1 + \cos \theta)^{n\alpha}}\right|_{0}^{\pi/2}\\
& = \frac{2}{n^2\alpha^2 - 1}.
\end{align*}
A lower bound on $C_1$ can be obtained analogously. As a result, we get
$$
\frac{2}{n^2\alpha^2 - 1}(n + 1)^{-\alpha} < C_1 <  \frac{2}{n^2\alpha^2 - 1}.
$$

It remains to determine upper and lower bounds on $C_2$. We apply the substitution $x = \cos\left(\theta/(n + 1)\right)$ and make use of the identity
$$
U_n\left(\cos\left(\frac{\theta}{n + 1}\right)\right) = \frac{\sin \theta}{\sin\left(\frac{\theta}{n + 1}\right)}
$$
as follows:
\begin{align*}
C_2
& = 2\int\limits_0^1\frac{dx}{|U_n(x)|^\alpha}\\
& = \frac{2}{n + 1}\int\limits_0^{\pi (n + 1)/2}\sin\left(\frac{\theta}{n + 1}\right)^{1 + \alpha}|\sin \theta|^{-\alpha}d\theta\\
& = \frac{2}{n + 1}\int\limits_0^{\pi (n + 1)/2}\sin\left(\frac{\theta}{n + 1}\right)^{1 + \alpha}d\theta + \varepsilon_n\\
& = 2\int\limits_0^{\pi/2}(\sin\theta)^{1 + \alpha} d\theta + \varepsilon_n\\
& = B\left(\frac{2 + \alpha}{2}, \frac{1}{2}\right) + \varepsilon_n,
\end{align*}
where the last equality follows from (\ref{eq:beta}) and
$$
\varepsilon_n = \frac{2}{n + 1}\int\limits_0^{\pi(n + 1)/2}\sin\left(\frac{\theta}{n + 1}\right)^{1 + \alpha}\left(|\sin \theta|^{-\alpha} - 1\right)d\theta.
$$
Since $\varepsilon_n > 0$, we obtain the lower bound $C_2 > B\left(\frac{2 + \alpha}{2}, \frac{1}{2}\right)$.

To find an upper bound on $C_2$, note that
\begin{align*}
\varepsilon_n
& \leq \frac{2}{n + 1}\int\limits_0^{\pi(n + 1)/2}\left(|\sin \theta|^{-\alpha} - 1\right)d\theta\\
& = 2\int\limits_0^{\pi/2}(\sin \theta)^{-\alpha} - \pi\\
& = B\left(\frac{1 - \alpha}{2}, \frac{1}{2}\right) - \pi,
\end{align*}
where the last equality follows from (\ref{eq:beta}). We conclude that
$$
B\left(\frac{2 + \alpha}{2}, \frac{1}{2}\right) < C_2 \leq B\left(\frac{2 + \alpha}{2}, \frac{1}{2}\right) + B\left(\frac{1 - \alpha}{2}, \frac{1}{2}\right) - \pi.
$$
\end{proof}

\section{Bean's Conjecture} \label{sec:bean}

We conclude this article with the proof of Proposition \ref{prop:bean}, which provides further theoretical evidence in support of Conjecture \ref{conj:bean}. 

\begin{prop} \label{prop:bean}
For any integer $n \geq 3$, $Q(T_n) \leq Q(S_n)$ and $Q(U_n) \leq Q(S_n)$. Furthermore,
$$
\lim\limits_{n \rightarrow \infty} Q(\Psi_n) = \lim\limits_{n \rightarrow \infty} Q(\Pi_n)  = \lim\limits_{n \rightarrow \infty} Q(T_n) = \lim\limits_{n \rightarrow \infty} Q(U_n) = \frac{16}{3}.
$$
\end{prop}

\begin{table}[t]
\centering
\begin{tabular}{c | c | c | c | c | c | c | c | c | c | c}
$n$ & $D_{\Psi_n}$ & $A_{\Psi_n}$ & $D_{\Pi_n}$ & $A_{\Pi_n}$ & $D_{S_n}$ & $A_{S_n}$ & $D_{T_n}$ & $A_{T_n}$ & $D_{U_n}$ & $A_{U_n}$\\
\hline
$3$ & 1 & $\infty$ & $2^23$ & $\infty$ & $2^23^3$ & $7.28585$ & $2^4 3^3$ & $5.78286$ & $2^{11}$ & $4.46217$\\
$4$ & $1$ & $\infty$ & $1$ & $\infty$ & $2^2$ & $10.4882$ & $2^{17}$ & $4.30008$ & $2^{16}5^2$ & $3.50332$\\
$5$ & $5$ & $\infty$ & $2^45^3$ & $5.78302$ & $2^{12}5^5$ & $4.55444$ & $2^{16}5^{5}$ & $3.78568$ & $2^{28}3^3$ & $3.19719$\\
$6$ & $1$ & $\infty$ & $2^23$ & $\infty$ & $2^{16}3^6$ & $5.29992$ & $2^{31}3^6$ & $3.52082$ & $2^{36}7^4$ & $3.04985$\\
$7$ & $7^2$ & $8.31171$ & $2^67^5$ & $5.38644$ & $2^{30}7^7$ & $3.99650$ & $2^{36}7^7$ & $3.35841$ & $2^{64}$ & $2.96434$\\
$8$ & $2^3$ & $\infty$ & $2^3$ & $\infty$ & $2^{24}$ & $6.48467$ & $2^{73}$ & $3.24832$ & $2^{64}3^{12}$ & $2.90894$\\
$9$ & $3^4$ & $7.64379$ & $2^63^9$ & $5.63543$ & $2^{56}3^{18}$ & $3.75495$ & $2^{64}3^{18}$ & $3.16867$ & $2^{88}5^7$ & $2.87035$
\end{tabular}
\caption{Invariants associated with $\Psi_n$, $\Pi_n$, $S_n$, $T_n$ and $U_n$ for $n = 3, 4, \ldots, 9$.}
\label{tab:table1}
\end{table}

\begin{proof}
Since $Q(F_n^*) = Q(S_n)$, we prove our statement with $F_n^*$ in place of $S_n$.

Let $n$ be an integer such that $n \geq 3$. According to Tran \cite{tran}, the discriminants of $T_n$ and $U_n$ are given by
\begin{equation} \label{eq:discriminants}
D_{T_n} = 2^{(n - 1)^2}n^n \quad \text{and} \quad D_{U_n} = 2^{n^2}(n + 1)^{n - 2}.
\end{equation}
Along with Theorems \ref{thm:main1} and \ref{thm:main2}, these equalities yield the upper bounds
\small
\begin{equation} \label{eq:QTn-bound}
2^{-\frac{n - 1}{n}}n^{-\frac{1}{n - 1}}Q(T_n) < \frac{8}{3} + \frac{2}{3}\left(\sqrt[n]{4} - 1\right) + B\left(\frac{1}{2} - \frac{1}{n}, \frac{1}{2}\right) - \pi
\end{equation}
\normalsize
and
\small
\begin{equation} \label{eq:QUn-bound}
2^{-\frac{n}{n - 1}}(n + 1)^{-\frac{n - 2}{n(n - 1)}} Q(U_n) < \frac{8}{3} + \left(B\left(1 + \frac{1}{n}, \frac{1}{2}\right) - 2\right) + B\left(\frac{1}{2} - \frac{1}{n}, \frac{1}{2}\right) - \pi
\end{equation}
\normalsize

Now, the values of $Q(F_n^*) = Q(S_n)$, $Q(T_n)$ and $Q(U_n)$ for $n = 3, 4, \ldots, 9$ are given in \mbox{Table \ref{tab:table2}}. The integrals associated with $A_{S_n}$, $A_{T_n}$ and $A_{U_n}$ were approximated with Mathematica. Notice that $Q(T_3) = Q(U_3) = Q(S_3) = 3B\left(\frac{1}{3}, \frac{1}{3}\right)$, because $S_3$, $T_3$ and $U_3$ are all equivalent under $\operatorname{GL}_2(\mathbb R)$ to \mbox{$xy(x - y)$}. From Table \ref{tab:table1} it is clear that the statement of \mbox{Corollary \ref{prop:bean}} holds for the aforementioned values of $n$.

For $n \geq 10$, the inequalities $Q(T_n) \leq Q(S_n)$ and $Q(U_n) \leq Q(S_n)$ can be easily verified by combining the formula (\ref{eq:QFn}) with the upper bounds (\ref{eq:QTn-bound}) and (\ref{eq:QUn-bound}).

The formulas for $D_{T_n}$ and $D_{U_n}$ given in (\ref{eq:discriminants}) imply
$$
\lim\limits_{n \rightarrow \infty}D_{T_n}^{1/n(n - 1)} = \lim\limits_{n \rightarrow \infty}D_{U_n}^{1/n(n - 1)} = 2.
$$
Combining this with the fact that $\lim\limits_{n \rightarrow \infty} A_{T_n} = \lim\limits_{n \rightarrow \infty} A_{U_n} = \frac{8}{3}$, we find that
$$
\lim\limits_{n \rightarrow \infty} Q(T_n) = \lim\limits_{n \rightarrow \infty} D_{T_n}^{\frac{1}{n(n - 1)}}A_{T_n} = \left(\lim\limits_{n \rightarrow \infty} D_{T_n}^{\frac{1}{n(n - 1)}}\right)\cdot\left(\lim\limits_{n \rightarrow \infty} A_{T_n}\right) = \frac{16}{3}.
$$
The same calculation applies with $Q(U_n)$ in place of $Q(T_n)$.

Next, we prove that $\lim\limits_{n \rightarrow \infty}Q(\Psi_n) = \frac{16}{3}$. For a positive integer $n$, let $\omega(n)$ denote the number of distinct prime divisors of $n$, and let \mbox{$n = p_1^{e_1} \cdots p_{\omega(n)}^{e_{\omega(n)}}$} be the prime factorization of $n$. According to Liang \cite{liang}, the discriminant $D_{K_n}$ of the number field $K_n = \mathbb Q\left(2\cos\left(\frac{2\pi}{n}\right)\right)$ can be computed according to the formula
\begin{equation} \label{eq:Dn}
D_{K_n} =
\begin{cases}
2^{(m-1)2^{m - 2} - 1} & \textrm{if $n = 2^m$, $m > 2$,}\\
p^{(mp^m - (m+1)p^{m-1} - 1)/2} & \textrm{if $n = p^m$ or $2p^m$, $p > 2$ prime,}\\
\left(\prod_{i = 1}^{\omega(n)}p_i^{e_i - 1/(p_i - 1)}\right)^{\frac{\varphi(n)}{2}} & \textrm{if $\omega(n) > 1, n \neq 2p^m$.}
\end{cases}
\end{equation}
It was also established by Liang that the ring of integers $\mathcal O_{K_n}$ of $K_n$ has a power integral basis, i.e., $\mathcal O_{K_n} = \mathbb Z\left[2\cos\left(\frac{2\pi}{n}\right)\right]$. Consequently, the discriminant $D_{\Psi_n}$ of $\Psi_n$ is equal to $D_{K_n}$. Since $1 \leq D_{\Psi_n} \leq n^{\varphi(n)/2}$, it follows from (\ref{eq:phi-lower-bound}) and Theorem \ref{thm:main3} that
$$
\lim\limits_{n \rightarrow \infty} Q(\Psi_n) = \frac{16}{3}.
$$
In view of Corollary \ref{cor:main3}, this result also implies \mbox{$\lim\limits_{n \rightarrow \infty}Q(\Pi_n) = \lim\limits_{n \rightarrow \infty}Q(\Psi_n) = \frac{16}{3}$}.
\end{proof}

\section*{Acknowledgements}

The author is grateful to Prof.\ Cameron L.\ Stewart for his numerous suggestions, to Patrick Naylor for many productive conversations, and to the anonymous reviewer for their valuable recommendations.

\end{document}